\documentclass{article}
\usepackage{graphicx} 
\usepackage[utf8]{inputenc}
\usepackage{amsmath}
\usepackage{amssymb}
\usepackage{amsthm}
\usepackage{amsfonts}
\usepackage{url}
\usepackage[all,knot,poly]{xy}
\usepackage{todonotes}
\theoremstyle{definition}
\newtheorem{definition}{Definition}

\newtheorem{proposition}{Proposition}
\newtheorem{lemma}{Lemma}
\newtheorem{remark}{Remark}

\usepackage{tikz-cd}

\newcommand\I{\overrightarrow{I}}
\newcommand\dipi[1]{\overrightarrow{\pi}_1(#1)}

\newcommand\dipin[1]{\overrightarrow{\pi}_n(#1)}

\newcommand\Real[0]{\mathbb{R}}
\newcommand\N[0]{\mathbb{N}}
\newcommand\look[1]{{\bf #1}}

\title{Directed homotopy modules}
\author{Eric Goubault}
\date{LIX, CNRS, École polytechnique, Institut Polytechnique de Paris, 91120 Palaiseau, France}

\begin{document}

\maketitle

\section{Introduction}

In this short note, we argue that directed homotopy can be given the structure of generalized modules, over particular monoids. This is part of a general attempt for refoundation of directed topology \cite{thebook,grandisbook}, which first step was given, for directed homology, in \cite{persmod}. 

For directed homology, the intuition in \cite{persmod} was that the Abelian homology groups of trace spaces between two points could be given the structure of a bimodule over a path algebra. In this note, we see the homology groups of trace spaces between two points as (generalized) bimodules over trace monoids. The bimodule structure will be natural by noting that trace monoids have the structure of absorption monoids, aka monoids in the category of pointed sets, thus very similar to rings, aka monoids in the category of Abelian groups. 

Among the nice consequences of this reformulation is that this paves the way towards putting directed homology and homotopy structures in the same general framework of framed bicategories in the sense of \cite{framed}: some of the properties we are proving in this note (in particular the restriction, extension and co-extension of scalars functor) go into  that direction. Still, the higher-categorical framework is left for a future article. 

\paragraph{Contents}

In Section \ref{sec:dspace} we recall the basic notions from directed homotopy theory, and in particular that of a directed space. Then in Section \ref{sec:absorb} we describe absorption monoids, and construct trace monoids of a directed space. These will play the role of (non-abelian) coefficient for directed homotopy modules. These non-abelian modules are introduced in Section \ref{sec:modules} and the definition of directed homotopy modules are given in Section \ref{sec:dirhommod}. As usual in homotopy theory, the 0th homotopy is not a module but a mere set: this is the case here as well, but for the first directed homotopy, which is a pointed set, but still with a (weak) module structure over the trace monoid, this is described in Section 
\ref{sec:pointedsets}. 
Finally, we show in Section \ref{sec:changes} that these dihomotopy modules behave like ordinary modules, with restriction and extension (and co-extension) of scalars functors, much in the spirit of what has been done in \cite{persmod} for directed homology. 


\section{Background on directed topology}

\label{sec:dspace}

The concept of a directed space was introduced in \cite{grandisbook}.    
Let $I=[0,1]$ denote the unit
segment with the topology inherited from $\Real$. 

\begin{definition}[\cite{grandisbook}]
A directed topological space, or a d-space is a pair  $(X,dX)$ consisting of a topological space $X$ equipped with
a subset $dX\subset X^I$ of continuous paths $p:I \rightarrow X$, called directed paths or
d-paths, satisfying three axioms: 
\begin{itemize}
\item every constant map $I\rightarrow X$ is directed;
\item $dX$ is closed under composition with continuous non-decreasing maps  $I\to I$;
\item $dX$ is closed under concatenation.
\end{itemize}
\end{definition}

We shall abbreviate the notation $(X,dX)$ to $X$. 

Note that for a d-space $X$, the space of d-paths $dX\subset X^I$ is a topological space, it is equipped with the compact-open topology. 
A continuous map $\varphi \ : I \rightarrow I$ is called a reparametrization if $\varphi(0) = 0$, $\varphi(1) = 1$ and if $\varphi$ is increasing, i.e. if $s\leq t \in I$ implies $\varphi(s) \leq \varphi(t)$.
Two directed paths $p, q \ : \I \rightarrow X$ are called reparametrization equivalent if
there exist reparametrizations $\varphi$, $\psi$ such that $p \circ \varphi = q \circ \psi$. The trace space $Tr(X)$ is the topological space of directed paths, quotiented by reparametrizations (see e.g. \cite{raussen2007reparametrizations}). For each $a$, $b$ in $X$, we call $Tr(X)(a,b)$ the sub-topological space of traces $[p]$ such that $p(0)=a$ and $p(1)=b$, i.e. of paths that begin at $a$ and end at $b$, modulo reparametrization. 

Note that there is a partially defined binary operation on directed paths of $X$, called the concatenation of directed paths, defined as follows:
$$p*q(t)=\left\{\begin{array}{ll}
p(2t) & 0\leq t \leq 1/2 \\
p(2t-1) & 1/2\leq t \leq 1
\end{array}\right.
$$
\noindent for $p$ a dipath in $X$ from $a$ to $b$ and $q$ a dipath in $X$ from $b$ to $c$, which carries over to traces, so that it becomes associative (with as neutral elements, all constant paths). 

A map $f: X\to Y$ between d-spaces $(X, dX)$ and $(Y, dY)$ is said to be {\it a d-map} if it is continuous and for any d-path $p\in dX$ 
the composition $f\circ p:I\to Y$ belongs to $dY$. In other words we require that $f$ preserves d-paths. We write $df~: dX \rightarrow dY$ for the induced map between directed paths
spaces. We denote by ${\cal D}$ the category of d-spaces. 

An isomorphism (also called dihomeomorphism) between directed spaces is a homeomorphism which is a directed map, and whose inverse is also a directed map. One of the ultimate goals of directed topology is to classify directed spaces up to dihomeomorphisms. 

Classical examples of d-spaces arise as geometric realizations of precubical sets, as found in e.g. the semantics of concurrent and distributed systems \cite{thebook}. We will use some of these classical examples to explain the directed homology theory we are designing in this article, hence we need to recap some of the basic definition on precubical sets, as well as their relation to d-spaces.

\section{Absorption monoids} 
\label{sec:absorb}


In this section, we define a particular monoidal structure, that has appeared in other contexts (see e.g. \cite{howie1995fundamentals}, with various names such as pointed, or absorption monoid, or monoid with zero), which will encode the quintessential aspects of traces, in directed spaces. 

\subsection{Basic definitions}

\begin{definition}
\label{def:absorptionmon}
Pointed (or absorption) monoids are monoids $(M,\times)$ with an absorbing element $0$, i.e., it obeys the following algebraic laws:
\begin{eqnarray}
m \times (m' \times m'') & = & (m\times m')\times m''\\
m \times 1 & = & m \\
1 \times m & = & m \\
m \times 0 & = & 0 \\
0 \times m & = & 0
\end{eqnarray}
\end{definition}

These come with suitable morphisms, to form the category of absorption monoids: 

\begin{definition}
A morphism $f$ of pointed monoids is a function $f: \ M \rightarrow M'$ such  that:
\begin{eqnarray}
f(m\times m') & = & f(m) \times f(m') \\
f(1) & = & 1 \\
f(0) & = & 0
\end{eqnarray}
\end{definition}

Some absorption monoids we will encounter have extra properties:

\begin{definition}
Pointed (or absorption) groups $G$ are absorption monoids $M$ such that $M\backslash \{0\}$ is a group. Pointed (or absorption) abelian groups $G$ are absorption monoids $M$ such that $M\backslash\{0\}$ is an abelian group. 
\end{definition}

We write $Mon_*$ for the category of absorption monoids, $Grp_*$ for the category of absorption groups and $Ab_*$ for the category of absorption abelian groups.  

\subsection{Categorical properties of absorption monoids}

\label{sec:catabsorp}
The category of absorption monoids is complete and co-complete, as is the category of monoids. This can easily be seen as follows for the latter case: the algebras of the free monoid monad on the category $Set$  of sets are monoids, and the Eilenberg-Moore category of the free monoid monad is the category of monoids. As well known, Eilenberg-Moore categories are complete when the base category (here $Set$) is complete. For the co-completeness, this follows from general theorems as well, which show in particular that any Eilenberg-Moore category over a monad on $Set$ (or any slice of it) is also co-complete. It is an easy exercise to carry this over to the free absorption monoid monad on the category of pointed sets $Set_*$ (left-adjoint to the forgetful functor from absorption monoids to $Set_*$ that points out the absorbing element within the set of elements of the monoid). 

To make things more concrete, we give some categorical constructions of interest below:

\paragraph{Products}

The product $(M,\times)$ of absorption monoids $(M_i,\times_i)$, $i\in I$ exists and has as underlying set $\mathop{\Pi}\limits_{i\in I} M_i$. The product of its elements $(m_i)_{i\in I}$ and $(n_j)_{j\in I}$ is given by:
$$
(m_i)_{i\in I}\times (n_j)_{j\in I} = (m_i\times_i n_i)_{i \in I}
$$
\noindent and $0_M=(0_{M_i})_{i\in I}$.

\paragraph{Coproducts}

The coproduct of $(M_i,\times_i)$, $i\in I$ exists and consists, set-theoreti\-cally, of the set of sequences $(m_i)_{i\in I}$ such that only a finite number of the $m_i$, $i\in I$ is not equal to the corresponding absorbing element $0_{M_i}$. The product of two such ``finite sequences'' is:
$$
(m_i)_{i\in I} \times (n_j)_{j\in J}=(m_i\times n_i)_{i\in I}
$$
\noindent which is indeed a ``finite sequence''. 

\paragraph{Quotients (for co-equalizers)}

The quotient of an absorption monoid $M$ by an absorption monoid $N\subseteq M$ is the absorption monoid $M/N$ which has as elements, classes $[m]$ of elements $m\in M$ under the (monoid) congruence generated by the relation
$n \sim 0$, for $n\in N$.


\subsection{Trace monoids of directed spaces}

Indeed, traces of a directed space give rise naturally to an absorption monoid. 

Let $X$ be a directed space. 

\begin{definition}
\label{def:monoidofpaths}
The \look{trace monoid} 
$T_X$ of $X$ is the absorption monoid whose elements are $t=[p] \in Tr(X)$ plus two elements that we denote by $0$ and $1$ with the following internal multiplication:
\begin{itemize}
\item If 
$[p]$, $[q]\in Tr(X)$, :
$$[p]\times [q]=
\left\{\begin{array}{ll}
[p*q]& \mbox{if $p$ and $q$ are composable} \\
0 & \mbox{otherwise} 
\end{array}
\right.$$
\item $[p]\times 0=0 \times [p]=0\times 0=0$
\item $[p]\times 1=1\times [p]=[p]$
\item $1\times 1=1$
\end{itemize}
\end{definition}

\section{Modules over absorption monoids}

\label{sec:modules}
\subsection{Basic definitions}

\begin{definition}
A (left) \look{module} $(G,T)$ in $Mon_*$ over absorption monoids is an absorption monoid $M$ together with an action of an absorption monoid $T$, i.e. a map
$\bullet: \ T\times M \rightarrow M$ such that, for all $t$, $t'$ in $T$ and $m$, $m'$ in $M$: 
\begin{itemize}
\item $(t\times t') \bullet m=t\bullet (t'\bullet m)$
\item $1 \bullet m = m$
\item $0 \bullet m = 0 = t \bullet 0$ 
\item $t\bullet (m \times m')=(t\bullet m) \times (t\bullet m')$
\end{itemize}
\end{definition}

We define similarly right modules and bimodules, that are both left and right modules.

\begin{definition}
The category of left modules in $Mon_*$ over $Mon_*$ has left modules $(M,T)$ as objects, and as morphisms, pairs $(f,h): \ (M,T) \rightarrow (M',T')$ where $f: \ M \rightarrow M'$ is an absorption monoid homomorphism, and $h: \ T \rightarrow T'$ is an absorption monoid homomorphism, such that: 
$$ f(t\bullet m)=h(t)\bullet f(m)
$$
We denote the category of left modules by $LMod$. 
\end{definition}

We have similar definition for the category $RMod$ of right modules and of bimodules $Mod$. 
The subcategories of left modules (resp. right modules, bimodules) over a specific absorption monoid $T$ is denoted by $LMod_T$ (resp. $RMod_T$ and $Mod_T$). 

\begin{remark}
Let $R$ be a ring. 
The \look{monoid algebra} $A[T]$ of an absorption monoid $T$ is the $R$-algebra whose elements $a$ are formal sums of elements of $T$, i.e. are of the form (for some $k \geq 0$, $t_i \in T$ and $r_i\in R$ for $i=1,\ldots,k$):
$$a=\sum\limits_{i=1}^k r_i t_i$$
\noindent with the following $R$-module structure:
\begin{itemize}
\item $\sum\limits_{i=1}^k r_i t_i + \sum\limits_{i=1}^k s_i t_i=\sum\limits_{i=1}^k (r_i+s_i) t_i$
\item $r.\sum\limits_{i=1}^k r_i t_i=\sum\limits_{i=1}^k (r.r_i) t_i$
\end{itemize}
and internal (algebra) multiplication: 
$$\left(\sum\limits_{i=1}^k r_i t_i\right)\times \left(\sum\limits_{j=1}^k s_j t_j\right)=\sum\limits_{i,j=1}^k (r_i.t_j) (t_i\times t_j)$$

Now, from a module $M$ in $Mon_*$ over absorption monoids, we can define an action of the monoid algebra $A[M]$ on the $R$-module generated by the abelianization $Ab(M)$. This will allow for making the link between standard modules and modules over absorption monoids, and get Hurewitz like theorems. 
\end{remark}

\subsection{Categorical properties of modules over absorption mo\-noids}

We note that 
the category of modules in $Mon_*$ over $Mon_*$ is a particular case of a module over the  monoidal category of absorption monoids, with the extra requirement that $0\bullet m=0=t\bullet 0$. As with classical categories of modules over categories with minimal categorical properties (such as completeness and co-completeness), the category $LMod$ is complete and co-complete. We give below only some of the basic categorical constructs. 

\paragraph{Products}

The product of $(M_i,T_i)$, $i\in I$ exists and is $(\prod\limits_{i\in I} M_i,\prod\limits_{i\in I} T_i)$ (where each product is taken in the category of absorption monoids, see Section \ref{sec:catabsorp}) with componentwise action of sequences $(t_i)_{i\in I} \in \prod\limits_{i\in I} T_i$ over sequences $(m_i)_{i\in I}\in \prod\limits_{i\in I} M_i$.

\paragraph{Co-products}

The coproduct of $(M_i,T_i)$ (with action $\bullet_i$ of $T_i$ on $M_i$), $i\in I$ exists and is composed of the coproduct of the $(M_i)_{i\in I}$ (as an absorption monoid, see Section \ref{sec:catabsorp}) with an action of the coproduct (as an absorption monoid) of the $(T_i)_{i\in I}$ with componentwise action of elements:
$$
(t_i)_{i\in I} \bullet (m_j)_{j\in I} = (t_i\bullet_i m_i)_{i\in I}
$$
\noindent (all sequences $(t_i)_{i\in I}$ and $(m_j)_{j\in I}$ being ``finite", the sequence $(t_i\bullet_i m_i)_{i\in I}$ is also a ``finite" sequence).

\paragraph{Quotients}

The quotient of a module $(M,T)$ by the sub-module $(N,T)$ over absorption monoids is the absorption monoid $M/N$ (see Section \ref{sec:catabsorp}) with the action of $T$ induced in this quotient by the action of $T$ on $M$. Indeed, as $T\bullet N \subseteq N$ ($(N,T)$ is a module), this carries over nicely to the corresponding equivalence classes. 



\section{Pointed sets and categories of modules over pointed sets}

\label{sec:pointedsets}

\begin{definition}
A pointed set is a set that has a distinguished element $*$ (therefore, cannot be empty). Morphisms of pointed sets are functions that preserve the distinguished element $*$. We denote by $Set_*$ the category of pointed sets.     
\end{definition}

\begin{definition}
A (left) \look{module} $(S,T)$ in $Set_*$ over absorption monoids is a pointed set $S$ together with an action of an  absorption monoid $T$, i.e. a map
$\bullet: \ T\times S \rightarrow S$ such that, for all $t$, $t'$ in $M$ and $g$, $g'$ in $S$: 
\begin{itemize}
\item $(t\times t') \bullet s=t\bullet (t'\bullet s)$
\item $1 \bullet s = s$
\item $0 \bullet s = * = t\bullet *$ 
\end{itemize}
\end{definition}

\begin{definition}
The category of left modules in $Set_*$ over $Mon_*$ has left modules $(S,T)$ as objects, and as morphisms, pairs $(f,h): \ (S,T) \rightarrow (S',T')$ where $f: \ S \rightarrow S'$ is a pointed set homomorphism, and $h: \ T \rightarrow T'$ is an absorption monoid homomorphism, such that: 
$$ f(t\bullet m)=h(t)\bullet f(m)
$$

The category of (resp. left, right, bi-) modules in $Set_*$ over absorption monoids is denoted by $LModSet$ (resp. $RModSet$, $ModSet$). 
\end{definition}

\begin{remark}
Left modules $M$ in $Set_*$ over absorption monoids $T$ can be seen as an encoding of a transition system, at least when the absorption monoid $T$ is a free absorption monoid over an alphabet $\Sigma$. In the latter case, the absorption monoid is the language monoid of the transition system with states $S=M\backslash \{*\}$, given by the transition relation $Trans \subseteq S \times \Sigma \times S$ such that $(s,\sigma,s')\in Trans$ if and only if $\sigma \bullet s=s'$ (hence is different from $*$). The action $\bullet$ extends the transition function so as to define the action of strings on states of the transition system. 
\end{remark}



    

\section{Directed Homotopy Modules}

\label{sec:dirhommod}
We now have all necessary structures to define directed homotopy modules. 
First, we need the following objects. 

Let $\I$ be the directed interval as in Section \ref{sec:dspace}, and $I$ be the directed space $[0,1]$ where all paths are directed. 
We recall the following: 
the standard simplex of dimension $n$ is $$
\Delta_n=\left\{(t_0,\ldots,t_n) \mid \forall i\in \{0,\ldots,n\}, \ t_i \geq 0 \mbox{ and } \sum\limits_{j=0}^n t_j=1\right\}
$$

Now we define traces of all dimensions in a directed space:

\begin{definition}
Let $X$ be a directed space, $i\geq 1$ an integer. 
We call $p$ an $i$-trace, or trace of dimension $i$ of $X$, any 
 continuous map 
$$p: \Delta_{i-1} \rightarrow Tr(X)$$ 
\noindent which is such that: 
\begin{itemize}
\item $p(t_0,\ldots,t_{i-1})(0)$ does not depend on $t_0,\ldots,t_{i-1}$ and we denote it by $s_p$ (``start of $p$")
\item and $p(t_0,\ldots,t_{i-1})(1)$ is constant as well, that we write as $t_p$ (``target of $p$")
\end{itemize}
We write $T_i(X)$ for the set of $i$-traces in $X$. We write $T_i(X)(a,b)$, $a$, $b \in X$, for the subset of $T_i(X)$ made of $i$-traces from $a$ to $b$. 
\end{definition}

Now, we observe that the $i$-traces of a directed space naturally form an (absorption) monoid: 

\begin{lemma}
\label{lem:Timod}
Let $X$ be a directed space and $i\geq 1$. The pointed set $T_i(X)\cup \{*\}$ is a module (both on the left and on the right), hence a bimodule, over the absorption monoid $T_X$. 
\end{lemma}


\begin{proof}
Let $t$ be a trace in $X$, a directed space, and $p: \ \Delta_{i-1} \rightarrow Tr(X)$ in $T_i(X)$. We define the action $t \bullet p$ of $t$ on $p$ to be the map from $\Delta_{i-1}\rightarrow Tr(X)$ with:
$$
t \bullet p = \left\{\begin{array}{llll}
u \in \Delta_{i-1} & \rightarrow & x \in \I \rightarrow \left\{\begin{array}{ll}
t(2x) & 0 \leq x \leq 1/2 \\
p(2x-1,u) & 1/2 \leq x \leq 1
\end{array}\right. & \mbox{if $t(1)=p(0)$} \\
0 & & & \mbox{otherwise}
\end{array}\right.
$$
This obviously defines an action from $T_X$ the absorption monoid of $X$, on $T_i(X)\cup \{*\}$, the pointed set of $i$-traces (augmented with the distinguished element $*$), by adding the requirement that $0\bullet p=*$. 

The action on the right $p\bullet t$ is similar, with a concatenation of path $t$ after the $i$-trace $p$. 
\end{proof}




Now, we recall that standard simplexes come with natural operators. 
For $n \in \N$, $n
\geq 1$ and $0 \leq k \leq n$, the $k$th ($n-1$)-face (inclusion) of the topological $n$-simplex is the subspace inclusion
$$\delta_k: \ \Delta_{n-1} \rightarrow \Delta_n$$
induced by the inclusion
$$(t_0,\ldots,t_{n-1}) \rightarrow (t_0,\ldots,t_{k-1},0,t_k,
\ldots, t_{n-1})$$
For $n \in \N$ and $0\leq k < n$, the 
$k$th degenerate 
$n$-simplex is the surjective map
$$
\sigma_k: \ \Delta_n \rightarrow \Delta_{n-1}$$
\noindent induced by the surjection: 
$$(t_0,\ldots,t_n)\rightarrow (t_0,\ldots,t_{k}+t_{k+1},\ldots,t_n)$$

Now, we are going to see that 
an $i$-trace of $X$ is a particular $(i-1)$-simplex of the trace space of $X$, for which we can define boundary and degeneracy maps, as usual: 

\begin{definition}
\label{def:boundaries}
Let $X$ be a directed space. We define:
\begin{itemize}
\item Maps $d_{j}$, $j=0,\ldots,i$ acting on $(i+1)$-traces 
$p: \ \Delta_i \rightarrow Tr(X)$, $i\geq 1$: 
$$d_j(p)=p\circ \delta_j$$
\item Maps $s_k$, $k=0,\ldots,i-1$ acting on $i$-traces 
$p: \Delta_{i-1} \rightarrow Tr(X)$, $i\geq 1$: 
$$s_k(p)=p\circ \sigma_k$$
\end{itemize}
\end{definition}

\begin{lemma}
\label{lem:boundariesstartend}
Maps $d_j$ defined in Definition \ref{def:boundaries} induce maps from $T_{i+1}(X)$ to $T_i(X)$. Similarly, maps $s_j$ induce maps from $T_i(X)$ to $T_{i+1}(X)$. 
These maps 
restrict to maps from $T_{i+1}(X)(a,b)$ to $T_i(X)(a,b)$ for any $a$, $b \in X$ (respectively, from $T_i(X)(a,b)$ to $T_{i+1}(X)(a,b)$). 
\end{lemma}

\begin{proof}
We have to check is that the image of the boundary operators (resp. degeneracy operators) give $i$-traces from $(i+1)$-traces (resp. $(i+1)$-traces from $i$-traces). It is due to the fact that, indeed, 
%
if $p$ is a $(i+1)$-trace from $s_p$ to $t_p$, $d_j(p)(s_0,\ldots,s_{i-1})(0)=p(d_j(s_0,\ldots,s_{i-1}))(0)=s_p$, the other computations being similar. This also proves that boundary (resp. degeneracy) operators restrict to maps from $T_{i+1}(X)(s_p,t_p)$ to $T_i(X)(s_p,t_p)$ (resp. from $T_i(X)(s_p,t_p)$ to $T_{i+1}(s_p,t_p)$). 
\end{proof}

\begin{lemma}
\label{lem:singsimpset}
Let $X$ be a directed space. The boundary and degeneracy operators of Definition \ref{def:boundaries} give the sequence $ST(X)=(T_{i+1}(X))_{i\geq 0}$ the structure of a simplicial set, that we denote by $(ST(X),d_*,s_*)$ when we need to make the simplicial structure explicit. 

This restricts to a simplicial set $ST(X)(a,b)$ of traces from $a$ to $b$ in $X$, which is the singular simplicial set of the trace space from $a$ to $b$, with the compact-open topology. 
\end{lemma}

\begin{proof}
The simplicial relations are direct consequences of the simplicial relations on $\delta_j$ and $\sigma_k$. The fact that $ST(X)(a,b)$ inherits the simplicial structure comes from Lemma \ref{lem:boundariesstartend}. The fact that this simplicial set is the singular simplicial set of the trace space from $a$ to $b$ is obvious. 
\end{proof}



In fact, the simplicial structure of $ST(X)$ marries nicely with the module structure seen in Lemma \ref{lem:Timod}:

\begin{proposition} 
\label{prop:Kan}
Let $X$ be directed space. The $(ST_*(X),d,s)$ can be given the structure of a Kan simplicial object in the category of modules of pointed sets, where $ST_*(X)$ denotes the graded (pointed) set $(T_i(X)\cup \{*\})_{i\in \N}$. 
\end{proposition}

\begin{proof}
Extending the boundary and degeneracy operators so that $d_k(*)=*$ and $s_j(*)=*$ makes $ST_*(X)$ makes the structure a simplicial (pointed) set, because of Lemma \ref{lem:singsimpset}. 

The fact that this defines a simplicial object in the category of modules of pointed sets is the consequence of the fact that, for all $t\in T_X$, the absorption monoid of $X$, $d_j$ and $s_k$ are morphisms of modules of pointed sets over $T_X$. 

Indeed, we compute for all $t \in T_X$ and $p \in T_{i+1}(X)$ such that $p(0,x_1,\ldots,x_i)=t(1)$:

$$\begin{array}{rcl}
\scriptstyle d_j(t\bullet p)(x_1,\ldots,x_i) & \scriptstyle = & \scriptstyle (t\bullet p) \circ \delta_j (x_1,\ldots,x_i)\\
& \scriptstyle = & \scriptstyle (t\bullet p) \circ \delta_j (x_1,\ldots,x_{j-1},0,x_j,\ldots,x_i)\\
& \scriptstyle = & \scriptstyle x \in \I \rightarrow \left\{\begin{array}{ll}
\scriptstyle t(2x) & \scriptstyle 0 \leq x \leq 1/2 \\
\scriptstyle p(2x-1,x_1,\ldots,x_{j-1},0,x_j,\ldots,x_i) & \scriptstyle 1/2 \leq x \leq 1
\end{array}\right.
\end{array}
$$
\noindent whereas:
$$
\begin{array}{rcl}
\scriptstyle (t\bullet d_j(p))(x_1,\ldots,x_i) & \scriptstyle = & \scriptstyle x \in \I \rightarrow \left\{\begin{array}{ll}
\scriptstyle t(2x) & \scriptstyle 0 \leq x \leq 1/2 \\
\scriptstyle d_j(p)(2x-1,x_1,\ldots,x_i) & \scriptstyle 1/2 \leq x \leq 1
\end{array}\right. \\
& = & \scriptstyle x \in \I \rightarrow \left\{\begin{array}{ll}
\scriptstyle t(2x) & \scriptstyle 0 \leq x \leq 1/2 \\
\scriptstyle p(2x-1,x_1,\ldots,x_{j-1},0,x_j,\ldots,x_i) & \scriptstyle 1/2 \leq x \leq 1
\end{array}\right.\\
& \scriptstyle = & \scriptstyle d_j(t\bullet p)(x_1,\ldots,x_i)
\end{array}
$$
\noindent and similarly for the degeneracy operators $s_j$. 



Now, this is a Kan simplicial module since as a simplicial set, $ST(X)$ is the singular simplicial set of a topological space ($Tr(X)$), which is Kan. Furthermore, it is easily seen that the Kan filler can be taken to be a module morphism.

\end{proof}




Once we have a Kan simplicial set in the category of modules of pointed sets over absorption monoids, 
we have a natural definition of a homotopy module as we recall below:

Let $t$ be a trace in $X$, then $\dipin{X,t}$ is the module whose underlying (pointed) set is made of equivalence classes modulo (classical) homotopy of traces $p \in T_n(X)$ (together with the distinguished element $*$) such that: 
\[
\begin{tikzcd}
\partial \Delta_{n-1} \arrow{r}{\epsilon} \arrow[swap]{d}{i} & \Delta_0 \arrow{d}{t} \\
\Delta_{n-1} \arrow{r}{p} & Tr(X)
\end{tikzcd}
\]
\noindent is a commutative diagram, where $i$ is the inclusion map of the boundary $\partial \Delta_{n-1}$ of $\Delta_{n-1}$ into $\Delta_{n-1}$ and $\epsilon$ is the unique map from $\partial \Delta_{n-1}$ to the point $\Delta_0$.

Now define $\dipin{X}$ to consist, set-theoretically, of the disjoint union of all $\dipin{X,t}$ over $t\in T_X$, quotiented 
by the identification of all distinguished points $*=t$ in $\dipin{X,t}$. 


For $n=1$, $\dipi{X}$ is a module of pointed sets over the absorption monoid of traces $T_X$ as, indeed (see Lemma \ref{lem:Timod}), for $t \in T_X$ and $p\in \dipin{X,t'}$, $t\bullet p$ is either $*$ or belongs to $\dipin{X,t\times t'}$:

\begin{lemma}
\label{lem:hommod1}
Let $X$ be a directed space. The fundamental dihomotopy module (or first homotopy module), is the module $(\dipi{X}_n,T_x)$ in $Set_*$ which has as elements $0$ and any $g\in \bigcup\limits_{a,b \in X} \pi_{0}(Tr(X)(a,b))$. 
\end{lemma}

\begin{proof}
As degeneracy and boundary operators respect endpoints of $i$-traces, Lemma \ref{lem:boundariesstartend}, this means that elements of $\dipi{X,t}$ have the same endpoints $a$, $b$ as $t$. Then, the homotopy relation between elements of $\dipi{X,t}$ is easily seen to correspond to components of $Tr(X)(a,b)$.
\end{proof}

And in higher-dimension, it is easy to see that: 
\begin{lemma}
\label{lemma:hommodn}
Let $X$ be a directed space, and $n \geq 2$. The $n$th dihomotopy module is the module $(\dipin{X},T_X)$ which has as elements $0$ and any $g \in \bigcup\limits_{a,b \in X} \pi_{n-1}(Tr(X)(a,b))$ with multiplication, for $g\in \pi_{n-1}(Tr(X)(a,b))$ and $g'\in \pi_{n-1}(Tr(X)(a',b'))$: 
$$g\times g'=
\left\{\begin{array}{ll}
gg' & \mbox{if $a=a'$ and $b=b'$} \\
0 & \mbox{otherwise} 
\end{array}
\right.$$
\noindent and actions of $t\in Tr(X)(a',b')$ on $[g]\in \pi_{n-1}(Tr(X)(a,b))$ ($g$ is an $n$-trace, representative of this class) is given by:
$$
t\bullet [g] = \left\{ \begin{array}{ll}
[t\bullet g] & \mbox{if $b'=a$} \\
0 & \mbox{otherwise}
\end{array}\right.
$$
\noindent when $t\bullet g$ in the formula above is given in Lemma \ref{lem:Timod}
\end{lemma}

\begin{proof}
Consider $p$ and $q$ two elements of $\dipin{X,t}$, $n \geq 2$ and elements of $T_n(X)$:
$$
v_i = \left\{\begin{array}{ll}
s_0\circ \ldots \circ s_0(t) & \mbox{if $0 \leq i \leq n-2$}\\
p & \mbox{if $i=n-1$} \\
q & \mbox{if $i=n+1$}
\end{array}\right.
$$
Now, as all boundary and degeneracy operators of $i$-traces respect the end points, Lemma \ref{lem:boundariesstartend}, this means that $p\times q$, this 
is easily seen to constitute a horn in $Tr(X)$ if $p$ and $q$ have the same endpoints $a$, $b$. In that case, by the Kan property, it has a filler which is an $n$-simplex $r: \ \Delta_n \rightarrow Tr(X)$ in the space of traces of $X$. 
We define an internal multiplication in $\dipin{X,t}$ by:
$$
p\times q = d_n(r)
$$
This multiplication is easily seen to be equal to the multiplication internal to group $\pi_n(Tr(X)(a,b)$. 

In case $p$ and $q$ do not have the same endpoints, we set $p\times q=0$. 

This multiplication is compatible with the action of $T_X$ over $T_i(X)$ (defined in Lemma \ref{lem:Timod}):
$$
t\bullet (p\times q) = (t\bullet p)\times (t\bullet q)
$$
\noindent as $ST(X)$ is a Kan simplicial object in the category of modules (Lemma \ref{prop:Kan}). 
\end{proof}
More precisely, $\dipin{X}$ is a bimodule in 
$Grp_*$ over the absorption monoid of traces for $n=2$, and a bimodule in $Ab_*$ over the absorption monoid of traces for $n\geq 3$.

Let $dTop_I$ be the wide subcategory of the category $dTop$ of directed spaces where we retain only injective morphisms. 
Finally, we have: 

\begin{lemma}
The homotopy modules constructions of Lemma  \ref{lem:hommodn} (resp. of Lemma  \ref{lem:hommod1}) define a functor from $dTop_I$ to $Mod$ (resp. from $dTop$ to $ModSet$). 
\end{lemma}

\begin{proof}
First, we show that the trace monoid is part of a functorial definition. Let ${\cal T}: \ dTop \rightarrow Mon_*$ which, to any $X \in dTop$ associates its absorption monoid of traces $T_X$ (Definition \ref{def:monoidofpaths}). Now, let $f: \ X \rightarrow Y$ be a morphism of directed spaces. Then we define ${\cal T}(f)$ to be such that:
$
{\cal T}(f)([p])=[f\circ p]
$ and ${\cal T}(f)(0)=0$. 

Indeed, if $p$ and $q$ are composable dipaths in $Tr(X)$, 
${\cal T}(f)([p]\times [q])={\cal T}(f)([p*q])=[f\circ (p*q)]=[(f\circ p)*(f \circ q)]=[f\circ p]\times [f\circ q]={\cal T}(f)(p)\times {\cal T}(f)(q)$. Also, because $f$ being injective, $p$ and $q$ are not composable is equivalent to $f\circ p$ and $f\circ q$ are not composable, the action of ${\cal T}(f)$ on $0$ is correct as well. 
\end{proof}

\section{Changes of coefficients functors}

\label{sec:changes}

Last but not least, the dihomotopy modules indeed behave like modules, we can change the base absorption module, i.e. the coefficients of the modules, which will be necessary to deal with exact sequences of dihomotopy modules (to be expanded in future work). 


\paragraph{Restriction of coefficients}

\begin{definition}
Let $l: \ T \rightarrow T'$ be a morphism of absorption monoids. We define the \look{restriction of coefficients} functor:
$$
l^*: \ LMod_{T'} \rightarrow LMod_T
$$
\noindent which, to any module $(M',T') \in LMod_{T'}$ (for which the left action of $T'$ on $M'$ is denoted by ${ }_{T'}\bullet$) associates the modules $l^*(M',T')$:
\begin{itemize}
\item which has as underlying absorption monoid $M'$
\item and has as action ${ }_T\bullet$ of $t\in T$ on $m\in M$:
$$
t{ }_T\bullet m = l(t){ }_{T'} \bullet m
$$
And for any module homomorphism $f: \ (M_1,T) \rightarrow (M_2,T) \in LMod_T$, 
$$
l^*(f)(m)=f(m)
$$
\end{itemize}
\end{definition}

This extends to right modules and bimodules. 

\paragraph{Extension of scalars}

\begin{definition}
Let $l: \ T \rightarrow T'$ be a morphism of absorption monoids. We define the \look{extension of coefficients} functor:
$$
l_!: \ LMod_T \rightarrow LMod_{T'}
$$
\noindent which, to any module $(M,T) \in LMod_T$ associates a module $(M',T')$ with:
\begin{itemize}
\item elements of $M'$ are finite non-empty products (in the corresponding absorption monoid) of classes $\langle t',m\rangle$ of pairs $(t',m)$ modulo the following equations, plus elements 0 and 1: 
\begin{eqnarray}
\langle t',t\bullet m\rangle = \langle t'\times l(t),m\rangle \label{eq:9}\\
\langle t',0\rangle=0 \\
\langle t',1\rangle=1 \mbox{, for $t'\neq 0$} \\
\langle 0,m\rangle=0 \\
\langle t',m_1\rangle\times \langle t',m_2\rangle =  \langle t',m_1 m_2\rangle
\end{eqnarray}
\item and action of $s' \in T'$ given by: 
$$
s' \bullet \langle t',m\rangle=\langle s'\times t',m\rangle
$$
\noindent (the action on products of classes $\langle t'_i,m_i\rangle$ being a consequence of the modules axioms)
\end{itemize}
The effect of $l_!$ on a morphism $f: \ (M_1,T) \rightarrow (M_2,T) \in LMod_T$ is:
$$
l_!(f)(\langle t',m\rangle)=\langle t',f(m)\rangle
$$
\noindent (the effect of $l_!(f)$ on products of classes $\langle t'_i,m_i\rangle$ being given by the fact that $l_!(f)$ is a morphism of modules, hence in particular commutes with its monoid internal multiplication)
\end{definition}

This is well defined indeed: 
\begin{itemize}
\item The action of $T'$ on $M'$ is well-behaved with respect to the equivalence class (Equation (\ref{eq:9}) in particular): 
$$
\begin{array}{rcl}
s'\bullet \langle t',t\bullet m\rangle & = & \langle s'\times t',t\bullet m\rangle \\
& = & \langle s'\times t'\times l(t),m\rangle \\
& = & s'\bullet \langle t'\times l(t),m\rangle
\end{array}
$$
\item The definition $l_!(f)$ is compatile with the equivalence class (Equation (\ref{eq:9}) in particular): 
$$
\begin{array}{rcl}
l_!(f,Id)(\langle t',t\bullet m\rangle) & = & \langle t',f(t \bullet m)\rangle \\ 
& = & \langle t',t\bullet f(m)\rangle \\
& = & \langle t' \times l(t), f(m)\rangle \\
& = & l_!(f,Id)(\langle t'\times l(t),m\rangle)
\end{array}
$$
\end{itemize}

This extends to right modules and bimodules. 

Now we have:

\begin{lemma}
For every $l: \ T \rightarrow T'$ morphism of absorption monoids, $l_!$ is left adjoint to $l^* $.    
\end{lemma}

\begin{proof}
We prove that $$Hom_{LMod_{T'}}(l_!(M,T),(M',T'))\equiv Hom_{LMod_T}((M,T),l^*(M',T'))$$

Let $f: \ l_!(M,T) \rightarrow (M',T')$ be a morphism in $LMod_{T'}$. We define: 
$$
\overline{f}: \ (M,T) \rightarrow l^*(M',T')
$$
\noindent by:
\begin{eqnarray}
\overline{f}(m) & = & f(\langle 1,m\rangle)
\end{eqnarray}

First, we prove this defines a morphism in $LMod_T$:

\begin{eqnarray}
\overline{f}(t\bullet m) & = & f(\langle 1,t\bullet m\rangle) \\
& = & f(\langle l(t),m\rangle) \\
& = & f(l(t) \bullet \langle 1,m\rangle) \\
& = & l(t) { }_{T'} \bullet f(\langle 1,m\rangle) \\
& = & l(t) { }_{T'} \bullet \overline{f}(m) \\
& = & t { }_T \bullet \overline{f}(m) 
\end{eqnarray}

Now let $h: \ (M,T) \rightarrow l^*(M',T')$ be a morphism in $LMod_T$. We define: 
$$ 
\underline{h}: \ l_!(M,T) \rightarrow (M',T')
$$
\noindent by
$$
\underline{h}(\langle t',m\rangle) = t'\bullet h(m)
$$
\noindent (and the effect of $\underline{h}$ on products of classes $\langle t'_i,m_i\rangle$ being the product of the action on individual classes, as it should, for being a morphism of modules). 

This defines indeed a morphism:
\begin{eqnarray*}
\underline{h}(s'\bullet \langle t',m\rangle) & = & \underline{h}(\langle s'\times t',m\rangle) \\
& = & (s' \times t')\bullet h(m) \\ 
& = & s' \bullet (t' \bullet h(m)) \\
& = & s' \bullet \underline{h}(\langle t',m\rangle)
\end{eqnarray*}

Now we prove that these two transforms, $f \rightarrow \overline{f}$ and $h \rightarrow \underline{h}$ are inverse of one another. 

Consider first:
\begin{eqnarray*}
\overline{(\underline{h})}(m) & = & \underline{h}(\langle 1,m\rangle) \\
& = & 1 \bullet h(m) \\
& = & h(m)
\end{eqnarray*}
\noindent and then: 
\begin{eqnarray*}
\underline{(\overline{h})}(\langle t',m\rangle) & = & t' \bullet \overline{h}(m) \\
& = & t' \bullet h(\langle 1,m\rangle) \\
& = & h(t' \bullet \langle 1,m\rangle) \\
& = & h(\langle t',m\rangle)
\end{eqnarray*}

The fact that $\overline{(.)}$ and $\underline{(.)}$ are natural (isomorphisms) is obvious. 
\end{proof}

\begin{lemma}
Let $l: \ T \rightarrow T'$ be a morphism of absorption monoids, and $(M,T)$ be a module in $Mon_*$ over $Mon_*$. Then if $M$ is an absorption group, then $l_!(M,T)$ is an absorption group. 
\end{lemma}

\begin{proof}
We note that elements of $l_!(M,T)$ are of the form $x=\langle t_1,m_1\rangle \langle t_2,m_2\rangle\ldots <t_n,m_n\rangle$. If $M$ is an absorption group, then we note that $\langle t_n,m_n^{-1}\rangle\langle t_{n-1},m_{n-1}^{-1}\rangle$ $\ldots \langle t_1,m_1^{-1}\rangle$ is the inverse of $x$. 
\end{proof}

\begin{remark}
There is no reason a priori, though, that if $M$ is Abelian, $l_!(M,T)$ is Abelian as well.     
\end{remark}

\paragraph{Co-extension of scalars}

\begin{definition}
Let $l: \ T \rightarrow T'$ be a morphism of absorption monoids. We define the \look{co-extension of coefficients} functor: 
$$
l_*: \ LMod_{T} \rightarrow LMod_{T'}
$$
\noindent which, to any module $(M,T) \in LMod_T$ associates a module $(M',T')$ with: 
\begin{itemize}
\item Elements of $M'$ are morphisms $g$ in $LMod_T$ from $(T',T)$ to $(M,T)$, with $(T',T) \in LMod_T$ with the action of $t\in T$ on $t'\in T'$ being defined by $t\bullet t'=l(t)\times_{T'} t'$. 
\item The multiplication of $g_1$ and $g_2$ in $M'$ is the componentwise multiplication: for all $t' \in T'$, 
$$
(g_1\times g_2)(t')=g_1(t')\times g_2(t')
$$
\item The action of $t'' \in T'$ on $g \in M'$ is, for all $t' \in T'$:
$$
(t''\bullet g)(t') = g(t' \times t'')
$$
\end{itemize}
The effect of $l_*$ on a morphism $f: \ M_1 \rightarrow M_2 \in LMod_T$ is: for all $t' \in T'$,  
$$
l_*(f)(g)(t') = f(g(t'))
$$
\end{definition}

\begin{lemma}
For every $l: \ T \rightarrow T'$ morphism of absorption monoids, $l_*$ is right adjoint to $l^* $.    
\end{lemma}

\begin{proof}
We prove that $$Hom_{LMod_{T}}(l^*(M',T'),(M,T))\equiv Hom_{LMod_{T'}}((M',T'),l_*(M,T))$$

Let $f: \ l^*(M',T') \rightarrow (M,T)$ be a morphism in $LMod_{T}$. We define: 
$$
\overline{f}: \ (M',T') \rightarrow l_*(M,T)
$$
\noindent by: for all $m'\in M'$, 
\begin{eqnarray}
\overline{f}(m') & = & t' \in T' \rightarrow f(t' \bullet m')
\end{eqnarray}

First, we prove this defines a morphism in $LMod_{T'}$: for all $t' \in T'$:

\begin{eqnarray}
\overline{f}(t''\bullet m')(t') & = &   f(t'\bullet (t''\bullet m')) \\
& = & f((t'\times t'') \bullet m') \\
& = & (t'\times t'') \bullet f(m')
\end{eqnarray}

Also, we have: 
\begin{eqnarray}
(t'' \bullet \overline{f}(m'))(t') & = &  \overline{f}(m')(t'\times t'') \\
& = & f(t'\times t'' \bullet m') \\
& = & (t'\times t'') \bullet f(m') \\
& = & \overline{f}(t''\bullet m')(t')
\end{eqnarray}

Conversely, let $k: \ (M',T') \rightarrow l_*(M,T)$ be a morphism in $LMod_{T'}$, we define: 
$$
\underline{k}: \ l^*(M',T') \rightarrow (M,T)
$$
\noindent in $LMod_T$ by, for $m'\in l^*(M',T')$, 
$$
\underline{k}(m') =  k(m')(1_{T'})
$$

This is easily seen as well to define a morphim in $LMod_T$.  


We check now that $\underline{.}$ and $\overline{.}$ define inverse transformations: 

\begin{eqnarray*}
\underline{(\overline{f})}(m') & = & \overline{f}(m')(1_{T'}) \\
& = & f(1_{T'} \bullet m') \\
& = & f(m')
\end{eqnarray*}
\noindent and: 
\begin{eqnarray*}
\overline{(\underline{k})}(m') & = & t' \in T' \rightarrow  \underline{k}(t'\bullet m') \\
& = & t' \rightarrow  k(t' \bullet m')(1_{T'}) \\
& = & t' \rightarrow t'\bullet k(m')(1_{T'}) \\
& = & t' \rightarrow k(m')(1_{T'}\bullet t') \\
& = & t' \rightarrow k(m')(t') \\
& = & k(m')
\end{eqnarray*}

The fact that these transformations are natural is obvious. 
\end{proof}

\begin{remark}
If $(M,T)$ is a module in $Grp_*$ over $Mon_*$, $l_*(L,T)$ is a module in $Grp_*$ as well.  
\end{remark}

\section{Conclusion}

\label{sec:conc}

This note paves the way towards a (directed) homotopical counterpart of \cite{persmod}. In a future article, we will develop the directed homotopical sequences as well as formal relations with \cite{persmod}.


\end{document}